\documentclass[a4paper,10pt, oneside, reqno]{amsart}
\pdfoutput=1

\usepackage[activate={true,nocompatibility},final,tracking=false,kerning=true,spacing=true,factor=1100,stretch=0,shrink=0]{microtype}

\usepackage[T1]{fontenc} 
\usepackage[utf8]{inputenc}
\usepackage[british]{babel}
\usepackage{mathrsfs}
\usepackage{amsfonts}
\usepackage{amsmath} 
\usepackage{amsthm}  
\usepackage{amssymb}
\usepackage{mathtools}
\usepackage{graphicx}
\usepackage{xcolor}
\usepackage{url}
\usepackage{cite}
\usepackage{xifthen}
\usepackage{stmaryrd}
\usepackage{centernot}
\usepackage{indentfirst}
\usepackage[hang,flushmargin]{footmisc}
\usepackage{geometry}
\usepackage{wasysym}
\usepackage{hyperref}
\hypersetup{colorlinks=false, pdfborder={0 0 0}}
\usepackage{todonotes}
\setuptodonotes{inline, color=pink}
\usepackage{cleveref}
\usepackage{tikz}
\usetikzlibrary{matrix,arrows,decorations.pathmorphing}

\newcommand{\dotieconcat}[2]{
  \text{\raisebox{.8ex}{$\smallfrown$}}%
}
\usepackage[inline]{enumitem}
\setlist[1]{labelindent=\parindent}
\setlist[enumerate,1]{label = \arabic*.,ref   = \arabic*}
\setlist[enumerate,2]{label = \alph*),ref   = \theenumi\alph*)}
\setlist[enumerate,3]{label = \roman*),ref   = \theenumii.\roman*}

\usepackage{orcidlink}

\newcommand{\from}{\colon}

\newcommand{\implica}{\rightarrow}

\renewcommand{\phi}{\varphi}
\renewcommand{\epsilon}{\varepsilon}
\renewcommand{\models}{\vDash}
\newcommand{\proves}{\vdash}
\newcommand{\restr}{\upharpoonright}
\newcommand{\monster}{\mathfrak U}
\newcommand{\bla}[4]{{#1}_{#2}#3\ldots#3{#1}_{#4}}
\newcommand{\blah}[4]{{#1}^{#2}#3\ldots#3{#1}^{#4}}
\newcommand{\qfeq}{\equiv^{\mathrm{qf}}}
\newcommand{\then}{\Longrightarrow}
\newcommand{\subspaceweight}{w}
\newcommand{\lincomb}[1]{{#1}^*}
\def\Kvs{\textnormal{$K$-vs}}

\DeclareMathOperator{\qftp}{qftp}
\DeclareMathOperator{\supp}{supp}
\DeclareMathOperator{\axes}{axes}
\DeclarePairedDelimiter{\set}{\{}{\}}
\DeclarePairedDelimiter{\abs}{\lvert}{\rvert}
\DeclarePairedDelimiter{\strgen}{\langle}{\rangle}
\theoremstyle{definition}
\newtheorem{defin}{Definition}[section]
\newtheorem{thm}[defin]{Theorem}
\newtheorem{pr}[defin]{Proposition}
\newtheorem{co}[defin]{Corollary}
\newtheorem{lemma}[defin]{Lemma}
\newtheorem{notation}[defin]{Notation}
\newtheorem{eg}[defin]{Example}
\newtheorem{rem}[defin]{Remark}

\let\oldqed\qedsymbol
\newcommand{\qedclaim}{\mbox{$\underset{\textsc{claim}}{\square}$}}

\let\qedsymbol\oldqed

\title{Vector spaces with a union of independent subspaces}

\author{Alessandro Berarducci\,\orcidlink{0000-0002-9927-1147}}
\author{Marcello Mamino\,\orcidlink{0000-0001-9037-5903}}
\author{Rosario Mennuni\,\orcidlink{0000-0003-2282-680X}}
\address{Dipartimento di Matematica, Universit\`a di Pisa, Largo Bruno Pontecorvo 5, 56127 Pisa, Italy}
\email{alessandro.berarducci@unipi.it}
\email{marcello.mamino@unipi.it}
\email{R.Mennuni@posteo.net}
\thanks{Supported by the Italian research project PRIN 2017, ``Mathematical logic: models, sets, computability'', Prot. 2017NWTM8RPRIN}
\date{}


\begin{document}
\begin{abstract}Motivated by the theory of locally definable groups, we study the theory of $K$-vector spaces with a predicate for the union $X$ of an infinite family of independent subspaces. We show that if $K$ is infinite then the theory is complete and admits quantifier elimination in the language of $K$-vector spaces with predicates for the $n$-fold sums of $X$ with itself. If $K$ is finite this is no longer true, but we still have that a natural completion is near-model-complete.
\end{abstract}
\maketitle
\section{Introduction}
We study the theory $T_K$ of vector spaces over a field $K$ with a predicate $X$ for the union of an infinite family of independent subspaces.  

We show (\Cref{thm:qe}) that, if $K$ is infinite, then $T_K$ is complete and admits quantifier elimination in the expansion of the language of vector spaces by predicates for the $n$-fold sums $X^n$ of $X$ with itself. From this, we deduce its total transcendence (\Cref{co:omegastab}). 
We also investigate (\Cref{sec:finitefields}) the case of finite $K$, where  completeness fails and there are completions which do not eliminate quantifiers in the language described above, although they are still near-model-complete in the language of vector spaces together with a predicate for $X$ (\Cref{rem:nmc}).

We were led to the study of this theory while considering certain questions on locally definable groups which arose in \cite{eleftheriou_definable_2012,eleftheriou_lattices_2013,berarducci_discrete_2013}. For further details, see \Cref{sec:locdef}.

 \section{The theory}
 \begin{defin}
   Let $K$ be an infinite field. Let $L_K$ be the union of the language $L_\Kvs\coloneqq\set{+,0,\lambda\cdot\mid \lambda \in K}$ of $K$-vector spaces with the family of unary predicates $\set{X^n\mid n\in \omega}$. Let $T_K$ be the theory axiomatised as follows.
   \begin{enumerate}
   \item The reduct to $L_\Kvs$ is a $K$-vector space.
   \item\label{ax:usubspc} The predicate $X\coloneqq X^1$ is closed under multiplication by every $\lambda\in K$. In other words, it is a union of subspaces.
   \item The predicates $X^n$ are interpreted as the $n$-fold sums $\set{\bla x0+{n-1}\mid x_j\in X}$ of $X$ with itself, with the convention that $X^0=\set{0}$.
   \item\label{ax:parallel} The \emph{parallelism} relation $x\parallel y\coloneqq x+y\in X$ is an equivalence relation on $X\setminus \set 0$; we call the union of an equivalence class with $\set 0$ an \emph{axis}.
   \item There are infinitely many axes.
   \item The axes are linearly independent: if $\bla Y0,n$ are pairwise distinct axes and $a_i\in Y_i\setminus \set 0$, then $\bla a0,n$ are linearly independent.
   \end{enumerate}
 \end{defin}
 Observe that, by axioms \ref{ax:usubspc} and \ref{ax:parallel}, each axis is a subspace.
 \begin{rem}
If $V$ is a $K$-vector space and $\set{V_i\mid i\in I}$ is an infinite family of independent subspaces, then  $(V, \bigcup_{i\in I} V_i)$ is a model of $T_K$, where $X$ is interpreted as $\bigcup_{i\in I} V_i$.   Conversely, every model of $T_K$ is of this form, by setting $\set{V_i\mid i\in I}$ to be the family of the axes.
\end{rem}
\begin{proof}
Left to right: independence ensures that if $x,y,x+y\in \bigcup_{i\in I} V_i$, then there must be $i\in I$ such that $x,y,x+y\in V_i$. From this, it is easily inferred that $\parallel$ is an equivalence relation on $X\setminus \set 0$, and that the set of axes is precisely $\set{V_i\mid  i\in I}$; the remaining axioms are easily checked to hold. Right to left is immediate.
\end{proof}
\begin{eg}
The $K$-vector space $K^\omega$ may be turned into a model of $T_K$ by interpreting $X$ as the set of elements with support of size at most $1$; in other words, $f\in X$ iff there is at most one $n\in \omega$ such that $f(n)\ne 0$. The axes are precisely the sets $V_i=\set{f\in K^\omega\mid \forall j\in \omega\; f(j)\ne0\implica j=i}$. Observe that the subspace $\bigcup_{n\in \omega} X^n$ of finite support elements is the direct sum of the $V_i$, while $K^\omega$ is their direct product.
 \end{eg}

  \begin{notation}
We denote by $M$ a model of $T_K$.  For $A\subseteq M$, we will write $\mathcal F(A)$ for the intersection of $A$ with the $\bigvee$-definable subspace $\bigcup_{n\in \omega} X^n(M)$. The subspace generated by a subset $B$ of a vector space will be denoted by $\strgen{B}$. If $a$ is a tuple, for instance $a\in M^{n}$, we denote its coordinates by using subscripts, starting at $0$, so $a=(\bla a0,{n-1})$, and its length (in this example, $n$) by $\abs a$. In order to avoid confusion, we write indices of a sequence of tuples as superscripts, as in $\blah a0,\ell\in M^n$.
If $a$ is a tuple, by $\strgen{a}$ we mean $\strgen{\set{a_i\mid i<\abs a}}$. All tuples we consider have finite length. 
\end{notation}
 Observe that $X(M)$ is the union of the axes of $M$, and that each axis is definable with parameters. We sometimes write $X$, $X^n$ in place of $X(M)$, $X^n(M)$ if $M$ is clear from context.
 While theories of pure vector spaces are known to be strongly minimal, and in particular stable, expanding them by an arbitrary unary predicate may destroy stability, and in fact even result in a theory with the independence property, see~\cite{braunfeld_characterizations_2021}. Nevertheless, this is not the case for $T_K$, whose models are in fact easily classified.
 \begin{thm}\label{thm:countmodels}
For every ordinal $\alpha$ such that $\abs K\le \aleph_\alpha$, there are at most $2^{\aleph_0+\abs \alpha}$ models of $T_K$ of size $\aleph_\alpha$.
In particular, (every completion of) $T_K$ is superstable.
 \end{thm}
 \begin{proof}
   Each $M\models T_K$ of size $\aleph_\alpha$ is determined up to isomorphism by $\operatorname{codim}(\mathcal F(M))$ and, for each positive cardinal $\kappa\le \aleph_\alpha$, the number of $\kappa$-dimensional axes of $M$.
   This information can be coded by a function from $\omega+\alpha$ to itself, and superstability follows by~\cite[Theorem~VIII.2.1]{shelah_classification_1990}.
\end{proof}

\section{Infinite fields}
In this section, $K$ is an infinite field.
We prove quantifier elimination for $T_K$, from which completeness and total transcendence will readily follow.
\begin{rem}
  If $a\in X^n\setminus X^{n-1}$,  there is a unique $n$-element set $\operatorname{supp}(a)\coloneqq\set{\bla m0,{n-1}}$ of pairwise nonparallel elements of $X\setminus\set 0$ with $a=\bla m0+{n-1}$. Call it the \emph{support} of $a$.
\end{rem}

 \begin{defin}
Fix $M\models T_K$.
\begin{enumerate}
\item The collection of axes of $M$ is denoted by  $\axes(M)$.
\item If $a\in \mathcal F(M)$ and $Y\in \axes(M)$, we define the \emph{projection} $\pi_Y(a)$ to be the unique element of $\supp(a)$ in $Y$ if one exists, and $0$ otherwise.
\item If $a\in \mathcal F(M)$, we define $\axes(a)$ to be $\set{Y\in \axes(M)\mid \pi_Y(a)\ne 0}$.
\item If $a\in \mathcal F(M)$,  we define the \emph{weight} $w(a)$ of $a$ to be the cardinality  $\abs{\axes(a)}=\abs{\supp(a)}$.
\item Similarly, for $A\subseteq \mathcal F(M)$, we define $\axes(A)$ to be the set $\bigcup_{a\in A} \axes(a) = \set{Y\in \axes(M)\mid \pi_Y(A)\ne \set{0}}$ and define the weight $\subspaceweight(A)$ of $A$ as the cardinality $\abs{\axes(A)}$. 
\item If $A\subseteq\mathcal F(M)$ we define $\hat A\coloneqq \strgen{\set{\pi_Y(A)\mid Y\in \axes(A)}}$. 
\end{enumerate}
\end{defin}

\begin{rem} \label{rem:fma} The following facts will be used without explicit mention. 
	\begin{enumerate}
		\item If $a\in X^n$ there are at most $n$ axes $Y$ with $\pi_Y(a)\ne 0$. In other words, $w(a)\le n$.
		\item For $a\in M$ and $\lambda \in K\setminus\set 0$, we have $\axes (a) = \axes (\lambda a)$. 
		\item We have $\axes(\lambda_0 a_0 + \ldots + \lambda_{n} a_{n}) \subseteq \bigcup_{i \leq n} \axes(a_i)$, and the equality holds if the supports of $\lambda_0 a_0, \ldots, \lambda_n a_{n}$ are disjoint.
		\item\label{point:disjsupp}   For every $a,b\in \mathcal F(M)\setminus\set 0$, there can be at most $w(b)$ many $\lambda\in K$ such that $\supp(a)\cap \supp(\lambda b)\ne \emptyset$. It follows that, if $K$ is infinite, $a\in \mathcal F (M)$, and $S\subseteq X$ is a finite set, then there is $\lambda \in K$ such that the support of $\lambda a$ is disjoint from $S$.  
		\item If $A_0, A_1$ are subspaces of $\mathcal F (M)$, then $\axes(A_0+A_1) = \axes(A_0) \cup \axes(A_1)$. 
                \item Clearly, $A\subseteq \hat A$ always holds.
                \item If $A\subseteq\mathcal F(M)$ is a subspace, then $\hat A=\bigoplus_{Y\in \axes(A)}\pi_Y(A)$.
	\end{enumerate}
\end{rem}

\begin{lemma}\label{lemma:usesKinfinite}Let $K$ be an infinite field.
 For every finite-dimensional subspace
    $A\subseteq \mathcal F(M)$ there is $a_\dagger\in A$ such that
    $\axes(A)=\axes(a_\dagger)$. In particular,
    $\subspaceweight(A)=\max_{a\in A} w(a)$, and $\subspaceweight(A)=n\in \omega$ if and only if $A\subseteq X^n$ and $A\centernot \subseteq X^{n-1}$.
\end{lemma}
\begin{proof}
  Since $K$ is infinite, point \ref{point:disjsupp} of \Cref{rem:fma} above tells us that 
for every $a,b\in A$ there is $\lambda\ne 0$ such that $\supp(a)\cap \supp(\lambda b)=\emptyset$, from which we deduce that $\axes(a-\lambda b)=\axes(a)\cup \axes(b)=\axes({\strgen{\set{a,b}}})$. 
The conclusion then follows by induction on $\dim A$.
\end{proof}
The assumption that $K$ is infinite cannot be removed, see \Cref{pr:finitefields}.

\begin{lemma}\label{lemma:findim}The following statements hold.
  \begin{enumerate}
  \item Every finite-dimensional subspace of  $\mathcal F(M)$  is contained in some $X^n(M)$.
  \item   If $A$ is a finite-dimensional subspace of $\mathcal F(M)$, then so is $\hat A$, and $\axes (A) = \axes (\hat{A})$. Moreover, for all axes $Y$, we have  $\pi_Y(A) = \pi_Y(\hat A)$. 
  \end{enumerate}
 \end{lemma}
 \begin{proof}
   If $A\subseteq M$ is a finite dimensional subspace, then $A$ has nonzero projection only on finitely many axes $Y_0, \ldots, Y_{n-1}$, and $\pi_{Y_i}(A)$ has finite dimension for each $i< n$, so also $\hat A = \bigoplus_{i< n} \pi_{Y_i}(A)$ has finite dimension. Moreover, $\pi_{Y_i}(A) \subseteq X$, so $A \subseteq \hat A \subseteq X^n$. The rest is clear. 
 \end{proof}
Below, we consider the bilinear map $\cdot\from K^n\times M^n\to M$ defined as $\lambda\cdot a\coloneqq \sum_{i< \abs a} \lambda_i a_i$.
 \begin{defin}
Let  $a=(\bla a0,{\abs a-1})$ be a tuple from $M$.
\begin{enumerate}
\item We set   $\lincomb{a}\from K^{\abs a}\to M$ to be the linear function $\lambda\mapsto \lambda\cdot a$. 
\item If $W\subseteq K^{\abs a}$, let $W\cdot a  \coloneqq \{\lambda \cdot a \mid \lambda \in W \} = a^*(W)$. 
\item If $\strgen{a}\subseteq \mathcal F(M)$, let $g_a$  be the function from the family of subspaces of $K^{\abs a}$ to $\omega$ with
  \[
    g_a(W)=\abs{\set{Y\in \axes(\strgen{a})\mid W=\ker \pi_Y\circ \lincomb{a}}}
  \]
\end{enumerate}
\end{defin}
Observe that, by definition, $g_a(W)$ is at most $\subspaceweight(\strgen{a})$. 

\begin{rem}
We have $W \subseteq \ker \pi_Y\circ \lincomb{a}$ if and only if $Y\not\in \axes (W \cdot a)$. If we also have, for all $\lambda \in K^{\abs a} \setminus W$, that $Y \in \axes (\lambda \cdot a)$,  then $W = \ker \pi_Y\circ \lincomb{a}$. Finally, note that if $Y \not\in \axes(\strgen{a})$, then $\ker \pi_Y\circ \lincomb{a} = K^{\abs a}$, so $\ker \pi_Y \circ \lincomb{a}$ is interesting only when $Y\in \axes(\strgen{a})$. 
\end{rem}

\begin{lemma}\label{lemma:gger}
  Suppose that $a$ is a tuple with $\strgen{a}\subseteq \mathcal F(M)$, and let $W$ be a subspace of $K^{\abs a}$. Then  $g_a(W)\ge r$ if and only, if for every finite set
  $\set{\blah \lambda 0,\ell}\subseteq K^{\abs a}\setminus W$, 
  \begin{equation}\label{eq:diffger}\tag{*}
    \abs*{\left(\bigcap_{i\le \ell} \axes(\lambda^i\cdot a)\right)\setminus \axes(W \cdot a)}\ge r
  \end{equation}
\end{lemma}
In other words,  \eqref{eq:diffger} says there are at least $r$ axes $\bla Y0,{r-1}$ such that for each $j<r$ we have $W\subseteq \ker (\pi_{Y_j}\circ \lincomb{a})$  and for every $i\le \ell$ we have $\lambda^i\notin \ker(\pi_{Y_j}\circ \lincomb{a})$.
\begin{proof}
  If $W=\ker \pi_Y\circ \lincomb{a}$ and $\lambda\in K^{\abs a}\setminus W$, then $Y$ 
  is an axis of $\lambda\cdot a$ but not of $W \cdot a$. It follows that $g_a(W)\ge r\then \eqref{eq:diffger}$. For the opposite implication we observe  
  that, since $\abs{\axes(\strgen{a})}$ is finite, there are $\lambda^0, \ldots, \lambda^{\ell} \in K^{\abs a} \setminus W$ such that $\bigcap_{i \leq \ell} \axes (\lambda^i \cdot a) = \bigcap_{\lambda \in K^{\abs a} \setminus W} \axes (\lambda \cdot a)$. It follows that 
   \eqref{eq:diffger} implies the existence of pairwise distinct axes $\bla Y0,{r-1}\notin \axes(W \cdot a)$ such that
  \(
Y_i\in\bigcap_{\lambda\in K^{\abs a}\setminus W} \axes(\lambda\cdot a)
  \),
hence such that $W= \ker \pi_{Y_i}\circ \lincomb{a}$.
\end{proof}

 \begin{lemma}\label{lemma:firme}
If $\strgen{a}\subseteq \mathcal F(M)$, then the quantifier-free type $\qftp(a)$ of~$a$ determines $g_a$.
 \end{lemma}
 \begin{proof}
An inspection of the language $L_K$ shows that knowing $\qftp(a)$ amounts precisely to knowing which linear combinations $\lambda\cdot a$ lie in which $X^n$. Because we are assuming $\strgen a\subseteq \mathcal F(M)$, this is the same as knowing all weights $w(\lambda\cdot a)$.  Now let $W\subseteq K^{\abs a}$ be a subspace. Since $K$ is infinite, by \Cref{lemma:usesKinfinite} there is $\lambda\in W$ such that $w(W\cdot a)=w(\lambda\cdot a)$. If $a \qfeq b$, then $w(\lambda \cdot a) =  w(\lambda \cdot b) \leq w(W \cdot b)$, so $w(W\cdot a) \leq w(W\cdot b)$ and by symmetry $w(W \cdot a) = w(W \cdot b)$. We have thus proved that $\qftp(a)$ determines the function sending a subspace $W\subseteq K^{\abs a}$ to $\subspaceweight(W \cdot a)$.

 If $W_0, W_1$ are subspaces of $K^{\abs a}$, then  $\axes((W_0+W_1)\cdot a)=\axes(W_0 \cdot a)\cup \axes(W_1 \cdot a)$. Therefore, for every $W\subseteq K^{\abs a}$ and every $\lambda^0,\ldots,\lambda^\ell\in K^{\abs a}\setminus W$, we have 
\[
  \abs*{
\left(\bigcup_{i\le \ell} \axes(\lambda^i\cdot a)\right)\setminus \axes(W \cdot a)
}= \subspaceweight(\strgen{W\cup \set{\lambda^0,\ldots,\lambda^\ell}}\cdot a)-\subspaceweight(W \cdot a)
\]
 By using induction and the inclusion-exclusion principle, it follows that $\qftp(a)$ also determines, for every subspace $W\subseteq K^{\abs a}$, every finite $\set{\lambda^0,\ldots,\lambda^\ell}\subseteq K^{\abs a}\setminus W$,  and every $r\in \omega$, whether condition~\eqref{eq:diffger} in Lemma \ref{lemma:gger} holds or not. By the aforementioned lemma, this information in turn determines $g_a$.
\end{proof}

 \begin{pr}\label{pr:extendtohat}
 Let $a = (a_0, \ldots, a_{n-1})$ and $b = (b_0, \ldots, b_{n-1})$ be tuples of the same length from $M,N\models T_K$ respectively, and denote by $A$ and $B$ the respective generated subspaces. Assume $a\qfeq b$, $A\subseteq \mathcal F(M)$, and $B\subseteq \mathcal F(N)$. Then the map $a_i\mapsto b_i$ extends to an isomorphism $\hat A\to \hat B$.
 \end{pr}
 \begin{proof}
   Since $a\qfeq b$, by \Cref{lemma:usesKinfinite} we have $\subspaceweight(A)=\subspaceweight(B)$.  Observe that
   \[
     \hat A\coloneqq\strgen{\set{\pi_{Y}(\lambda \cdot a)\mid Y\in \axes(A), \lambda\in K^{\abs a}}}
   \]
   and similarly for $\hat B$.    By \Cref{lemma:firme} we have $g_a=g_b$, hence for each subspace $W$ of $K^{\abs a}$ there is a bijection between the set of axes $Y$ of $A$ with $\ker \pi_Y \circ \lincomb{a} = W$ and the set of axes $Z$ of $B$ with $\ker \pi_Z \circ \lincomb{b} = W$. 
Putting together these bijections as $W$ varies, we obtain a bijection $\sigma\from \axes(A)\to \axes(B)$ 
with $\ker \pi_Y \circ \lincomb{a} = \ker \pi_{\sigma(Y)} \circ \lincomb{b}$. 
This implies that the map $\pi_{Y}(\lambda \cdot a)\mapsto \pi_{\sigma(Y)}(\lambda \cdot b)$ is well-defined; let $h\from \hat A\to \hat B$ be its  linear extension. By construction, each restriction $h\restr \pi_Y(\hat A)$  is an isomorphism onto $\pi_{\sigma(Y)}(\hat B)$, and since $\hat A=\bigoplus_{Y\in \axes(A)} \pi_Y(\hat A)$ and $\hat B=\bigoplus_{Y\in \axes(A)}\pi_{\sigma(Y)}(\hat B)$, the map $h$ is an isomorphism. Moreover, each coordinate $a_i$ of $a$ is the sum of its projections on the axes of $A$, and analogously for $B$. Let us also observe that if $\lambda \in K^{\abs a}$ is the $i$-th vector of the standard base, then $a_i = \lambda \cdot a$ and $b_i = \lambda \cdot b$. Hence $h$ extends $a_i\mapsto b_i$ and we are done.
 \end{proof}
 \begin{lemma}\label{lemma:ida}
   In every $\omega$-saturated $N\models T_K$, every axis is infinite-dimensional, and the codimension of  $\mathcal F(N)$ is infinite.
 \end{lemma}
 \begin{proof}
   Both statements are proven by easy compactness arguments, the first one using that every axis is infinite, and the second one using that there are infinitely many axes.
 \end{proof}
 \begin{thm}\label{thm:qe}
For every infinite  $K$, the theory $T_K$ eliminates quantifiers in $L_K$ and is complete.
 \end{thm}
 \begin{proof}
  The  vector space $\set{0}$, with $0$ in every $X^n$, is a substructure  embedding in every  $M\models T_K$, hence completeness is a consequence of quantifier elimination and it suffices to prove the latter. To this end, we show that if $M, N\models T_K$ are  $\omega$-saturated, then the family of partial isomorphisms between finitely generated substructures $A$ of $M$ and $B$ of $N$ has the back-and-forth property. 
 
  Let $f\from A\to B$ be a partial isomorphism as above, and let $a\in M\setminus A$. We need to extend $f$ to an isomorphism of finitely generated substructures which includes $a$ in its domain. 
  
  To this aim fix a complement $\mathcal I(A)$ of $\mathcal F(A)$ in $A$, observe that $\mathcal I(A)\cap \mathcal F(M)=\set 0$, and let $\mathcal I (M)$ be a complement of $\mathcal F (M)$ which includes $\mathcal I (A)$. Then let $\mathcal I(B)\coloneqq f(\mathcal I(A))$ and let $\mathcal I (N)$ be a complement of $\mathcal F (N)$ which includes $\mathcal I (B)$. 
  
  It suffices to deal with the cases $a \in \mathcal I (M)$ and $a \in \mathcal F (M)$. 
  
 Suppose first that $a\in \mathcal I (M)$. By \Cref{lemma:ida} the codimension of $\mathcal F (N)$ in $N$ is infinite, so we can choose $b\in \mathcal I (N) \setminus \mathcal I(B)$ and extend $f:A\to B$ to a linear map which sends $a$ to $b$. This yields the required isomorphism since $\mathcal F (\strgen {A \cup \{a\}}) = \mathcal F (A)$ and $\mathcal F (\strgen {B \cup \{b\}}) = \mathcal F (B)$. 
 
 Now suppose that $a\in \mathcal F (M)$.  By \Cref{pr:extendtohat}, we may extend $f$ to an isomorphism
  \[
  \hat f\from \widehat{\mathcal F(A)}\oplus \mathcal I(A)\to \widehat{\mathcal F(B)}\oplus \mathcal I(B)
  \]
  so we can assume that $a\in \mathcal F (M) \setminus \widehat {\mathcal F (A)}$. 
   Write $a$ as the sum of the elements of its support, that is $a=\sum_{i<w(a)} m_i$, with each $m_i\in X(M)$ and $m_i\centernot\parallel m_j$ if $i\ne j$. It is enough to extend the isomorphism $\hat f$ to the elements $m_i$, so we may assume that $a \in X(M)$ and $a\not \in \widehat{\mathcal F (A)}$.  Since $a\in X(M)\setminus \set 0$, the set $\axes(a)$ has a unique element; call it $Y$.  We distinguish two cases.
  
Suppose $Y \in \axes (\mathcal F (A))$. Then there is $c\in X(M) \cap \widehat{\mathcal F (A)}$ such that $a\parallel c$. By \Cref{lemma:ida} each axis of $N$ has infinite dimension, so we may find  $b \in X(N)\setminus \widehat{\mathcal F(B)}$ such that $b \parallel \hat f(c)$ and we extend the isomorphism by sending $a$ to $b$. 

Suppose now that $Y\not \in \axes(\mathcal F (A))$. The axioms of $T_K$ ensure that $N$ has infinitely many axes, hence we may find  $b \in X(N)$ such that $b\not\in \axes(\mathcal F (B))$ and we extend the isomorphism by sending $a$ to $b$. 
\end{proof}

\begin{co}\label{co:omegastab}
  For every infinite field $K$, the theory $T_K$ is totally transcendental.
\end{co}
\begin{proof}
  It is well-known that a theory is totally transcendental if and only if each of its reducts to a countable sublanguage is $\omega$-stable. Therefore, it suffices to prove that, if $K$ is countably infinite and $M\models T_K$ is countable, then so is $S_1(M)$. By quantifier elimination and the axioms of $T_K$, every $p(x)\in S_1(M)$ is determined by boolean combinations of formulas of the form  or $x-m\in X^n$, for $m\in M$ and $n\in \omega$ (when $n=0$, this is the same as $x-m=0$).

  Clearly, there is a unique $1$-type extending $\set{x-m\notin X^n\mid n\in \omega, m\in M}$. It is therefore enough to count those $p(x)$ containing, for some $n\in \omega$ and $m\in M$,  the formula $x-m\in X^n$.   Since each $h_m\coloneqq y\mapsto y-m$ is a definable bijection, and because there are only countably many $h_m$, up to translating we may assume that $p(x)\proves x\in X^n$ and that $n$ is minimal such, that is, for all $m\in M$ and $n_0<n$ we have $p(x)\proves x-m\notin X^{n_0}$. In other words, it suffices to show that, for each $n\in \omega$, there are only countably many types $p(x)$ such that $p(x)\proves \set{x\in X^n}\cup\set{x-m\notin X^{n_0}\mid n_0<n}$.

  Let $p$, $n$ be as above. If $n=0$, then  $p(x)\proves x=0$.  For $n=1$, we argue as follows.  Let $m\in X(M)\setminus\set0$, and let $Y$ be the corresponding axis; view it as a definable set,  defined by $(x\in X)\land (x+m\in X)$. Again by quantifier elimination, the structure induced on $Y$ is that of a pure $K$-vector space, and $Y$ is stably embedded because $T_K$ is stable by \Cref{thm:countmodels}. Hence, $Y$ is a strongly minimal set, and there is a unique nonrealised type concentrating on $Y$, call it $p_Y$. Since $M$ is countable, there are only countably many axes in $M$, hence only countably many such $p_Y$. Now, for a type $p(x)$ satisfying the assumptions above with $n=1$, there are only two possibilities: either there is $Y\in \axes(M)$ such that $p=p_Y$, or $p$ is the uniquely determined type of an element of  $X$ in a new axis; more precisely, this is the unique type as above which furthermore satisfies $p(x)\proves \set{x+m\notin X\mid m\in X(M)}$; call it $p_{\mathrm{na}}$.

  For $n\ge 2$,  in some monster model $\monster\models T_K$, fix $a\models p$, and write $a=\bla a0+{n-1}$ as the sum of the elements in its support. Minimality of $n$ implies that, for each $i<n$, either there is $Y_i\in \axes(M)$ such that $a_i\models p_{Y_i}$, or  $a_i\models p_{\mathrm{na}}$. Up to a permutation of the coordinates we may assume that, for a suitable $k<n$,  the first case happens for $i<k$, and the second one for $k\le i <n$. Again by quantifier elimination, there is a unique $n$-type $q(y)$ such that $q(y)\proves \bigcup_{i<k} p_{Y_i}(y_i) \cup  \bigcup_{k\le i<n} p_{\mathrm{na}}(y_i) \cup    \set*{\bigwedge_{k\le i<j<n} y_i\centernot\parallel y_j}$.
Because there are only countably many axes in $M$, there are only countably many choices for such a $q(y)$. Since $p(x)$ is the pushforward of $q(y)$ along the definable function $y\mapsto \sum_{i<n}y_i$, it is uniquely determined by $q$.
\end{proof}

\section{The finite field case}\label{sec:finitefields}
If $K$ is finite, then the number of axes of a given dimension becomes first-order expressible, hence $T_K$ is incomplete. Since there is a structure that embeds in every model of $T_K$, namely the vector space  $\set{0}$ with $0$ belonging to every $X^n$, it follows that $T_K$ cannot eliminate quantifiers in $L_K$. A natural completion to consider is $T_{K,\infty}$, obtained by requiring that every axis is infinite-dimensional. Even $T_{K,\infty}$ does not eliminate  quantifiers in $L_K$, since the conclusion of \Cref{lemma:usesKinfinite} fails, as shown below. 
 \begin{pr}\label{pr:finitefields}
If $K$ is finite and $M\models T_{K,\infty}$, there are $a,b\in M^2$ with $a\qfeq b$, with both subspaces $\strgen{a}$ and $\strgen{b}$ included in $X^{\abs K}(M)$, but with  $\subspaceweight(\strgen{a})=\abs K<\abs K+1= \subspaceweight(\strgen{b})$.
\end{pr}
\begin{proof}
  For every $i<\abs K$, choose  linearly independent  $m_i\parallel m_i'$ in $X$ in such a way that if $i\ne j$ then $m_i\centernot\parallel m_j$.  Choose $m\in X$ with $m\centernot\parallel m_i$, fix an enumeration $(\lambda_i)_{1\le i\le \abs K-1}$ of $K\setminus \set 0$, and set
\[
    a_0\coloneqq \sum_{i< \abs K}m_i   \qquad  a_1\coloneqq \sum_{i< \abs K}m_i'\qquad
    b_0\coloneqq a_0  \qquad   b_1\coloneqq m+\sum_{1\le i< \abs K-1}\lambda_i m_i
  \]
Let $a=(a_0, a_1)$ and $b=(b_0, b_1)$. By definition, for $x\in \set{a,b}$, we have $w(x_0)=w(x_1)=\abs K$. Moreover, an easy computation shows that, for every $\lambda,\mu\in K\setminus \set 0$, we have $w(\lambda x_0+\mu x_1)=\abs K$, and it follows that $\strgen a, \strgen b\subseteq X^{\abs K}$, hence that $a\qfeq b$. Nevertheless, by construction $\axes(\strgen a)=\set{\axes(m_i)\mid i<\abs K}$, while $\axes(\strgen b)=\axes(\strgen a)\cup \axes(m)$, and since $m$ and the $m_i$ are all in $X$ we have $\subspaceweight(\strgen{a})=\abs K<\abs K+1=\subspaceweight(\strgen{b})$.
\end{proof}
\begin{rem}\label{rem:nmc}
For finite $K$, the proof of  \Cref{thm:qe} may be adapted to prove quantifier elimination for $T_{K, \infty}$ in the expansion of $L_K$ by predicates $P_{k,n,i}(\bla a0,k)$ such that, if $A\coloneqq\strgen{\bla a0,k}\subseteq X^n$, then $P_{k,n,i}$ codes the isomorphism type of $\hat A$ (with $i$ ranging over the possible isomorphism types). These predicates are $\exists$-definable in $L_K$, and in fact already in $L_\Kvs\cup \set X$ hence, in this language, $T_{K, \infty}$ is near-model-complete, that is, every formula is equivalent to a boolean combination of existential ones.
\end{rem}

\section{Locally definable groups}\label{sec:locdef}

The original motivation for considering the theory $T_K$ comes from the following question in \cite{eleftheriou_definable_2012,eleftheriou_lattices_2013}. 
Let $G$ be a locally definable connected abelian group in an o-minimal structure, and assume that $G$ is generated by a definable subset. Is $G$ a cover of a definable group? Recall that a group is \emph{locally definable} iff both its domain and the graph of its multiplication are countable unions of definable sets. 
 A subset $A$ of a locally definable set $Z$ is \emph{discrete} iff, for every definable $Y\subseteq Z$, the intersection $A\cap Y$ is finite.
In the cited papers it is shown that the answer is positive if and only if the following two conditions hold. 
\begin{enumerate}
    \item\label{point:maxk} There is a maximal $k\in \omega$ such that $G$ has a discrete subgroup isomorphic to ${\mathbb Z}^k$.
    \item\label{point:kne0} If $G$ is not already definable, then $k\neq 0$.
\end{enumerate}

By \cite{berarducci_discrete_2013} condition~\ref{point:maxk} is always satisfied, but  it remains open whether~\ref{point:kne0} holds. This led us to consider the following question. Let $G$ be a locally definable, definably generated abelian group in a first order structure and assume that $G$ is not definable. Does $G$ have a discrete subgroup isomorphic to $\mathbb Z$? Any model of the theory $T_K$ provides a negative answer in the stable context.  

\begin{pr}
In every model of $T_K$, the locally definable group $(\bigcup_{n\in \omega} X^n, +)$ is definably generated, but it contains no discrete copy of $\mathbb Z$.
\end{pr}
\begin{proof}
  Clearly $X$ is a definable generating set.
Now, consider a copy of~$\mathbb Z$, namely a subgroup of the form
$\mathbb Z\cdot a$.
Since $a\in\bigcup_{n\in \omega} X^n$, we have $a\in X^n$ for some~$n$. Since $X$ is closed under scalar multiplication, this
implies that $\mathbb Z\cdot a$ is contained in the definable set $X^n$,
and thus it cannot be discrete.
\end{proof}

\end{document}